\documentclass[12pt,leqno]{amsart}
\usepackage[dvipsnames]{color}

\usepackage{amsfonts,amssymb,amsmath,amscd,amstext}
\numberwithin{equation}{section}
\usepackage{amsthm}
\usepackage[colorlinks=true,linkcolor=teal,citecolor=purple]{hyperref}
\usepackage[utf8]{inputenc}
\usepackage{graphicx}
\usepackage{changes}
\usepackage{comment}
\usepackage{mathtools}
\usepackage{tikz-cd}
\usepackage[noabbrev,capitalize,nameinlink]{cleveref}
\usepackage{mathdots}
\usepackage[a4paper,top=2.3cm,bottom=2.3cm,left=2cm,right=2cm]{geometry}

\usepackage{ulem}
\usepackage{soul}

\renewcommand{\leq}{\leqslant}

\newcommand{\R}{{\bf R}}
\newcommand{\N}{{\bf N}}

\newcommand{\winf}[1]{W^{1,\infty}(#1)}

\newcommand{\X}{X}
\newcommand{\rr}{{\mathbf{R}}}

\newcommand{\Om}{\Omega}

\newcommand{\mA}{\mathcal{A}}

\newcommand{\scu}{\to }
\newcommand{\lp}[1]{L^{p}(#1)}

\renewcommand{\so}[1]{W^{1,p}(#1)}
\newcommand{\soz}[1]{W_0^{1,p}(#1)}
\newcommand{\sol}[1]{W^{1,p}_{loc}(#1)}
\newcommand{\anso}[1]{W^{1,p}_X(#1)}
\newcommand{\ansol}[1]{W^{1,p}_{X,loc}(#1)}

\newcommand{\Cm}{\C_{P}}

\newcommand{\lul}{L^1_{loc}(\Omega)}
\newcommand{\C}{\mathcal{C}}

\newcommand{\average}{{\mathchoice {\kern1ex\vcenter{\hrule height.4pt
				width 6pt
				depth0pt} \kern-9.7pt} {\kern1ex\vcenter{\hrule height.4pt width 4.3pt
				depth0pt}
			\kern-7pt} {} {} }}

\usepackage{amssymb,epsfig,xcolor,upgreek, cancel,enumitem}

\newfont{\bbten}{bbold12}

\definechangesauthor[color=red]{J}
\definecolor{champagne}{rgb}{0.97, 0.91, 0.81}

\definecolor{asparagus}{rgb}{0.53, 0.66, 0.42}

\DeclareMathOperator{\divv}{div}

\DeclareMathOperator{\Lip}{Lip}

\DeclareMathOperator{\im}{Im}
\newtheorem{theorem}{Theorem}[section]
\newtheorem{proposition}[theorem]{Proposition}
\newtheorem{definition}[theorem]{Definition}

\newtheorem{corollary}[theorem]{Corollary}
\theoremstyle{definition}
\newtheorem{example}[theorem]{Example}
\newtheorem{remark}[theorem]{Remark}
\definechangesauthor[name=Simone, color=blue]{S}

\title[A survey on anisotropic integral representation results]{A survey on anisotropic integral representation results}

\author[S.\ Verzellesi]{Simone Verzellesi}

\address[S.\ Verzellesi]{ Dipartimento di Matematica ``Tullio Levi-Civita'', Universit\`a di Padova,\newline \indent via Trieste 63 35121, Padova - Italy}
\email{simone.verzellesi@unipd.it}

\thanks{ The author is member of the Istituto Nazionale di Alta Matematica (INdAM), Gruppo Nazionale per l'Analisi Matematica, la Probabilità e le loro Applicazioni (GNAMPA). The author is supported by MIUR-PRIN 2022 Project \emph{Regularity problems in sub-Riemannian structures}  (Grant Code: 2022F4F2LH), and by INdAM-GNAMPA 2025 Project \emph{Structure of sub-Riemannian hypersurfaces in Heisenberg groups}, (Grant Code: CUP\_ES324001950001).
}

\begin{document}
\begin{abstract}
In this note we review some recent results concerning integral representation properties of local functionals driven by Lipschitz continuous anisotropies.
\end{abstract}
\maketitle


\section{Introduction}
A \emph{local functional} is a functional $F$ of two variables:
 a function $u$ belonging to a chosen functional space, a set $A$ belonging to a chosen family of sets.
The adjective \emph{local} refers to the fact that $F(u,A)$ depends only on the values that $u$ takes on $A$. An important class of local functionals is that of \emph{integral functionals}, e.g. functionals of the form
\begin{equation*}
    (u,A)\mapsto\int_A f(x,u,Du)\,dx,
\end{equation*}
where, say, $A$ is open and $u\in W^{1,p}(A)$. Integral functionals constitute the cornerstone of the modern calculus of variations, and describing their importance clearly goes beyond our scope. We refer to \cite{MR0990890,FonsLeonLp,doi:10.1142/5002} for complete introductions to this topic. In the study of optimization problems for local functionals, it is now an established practice to rely on the tool of \emph{$\Gamma$-convergence}, as introduced by  E. De Giorgi and T. Franzoni \cite{MR0375037,MR0448194}.
For an exhaustive introduction to this topic, we refer to the monographs \cite{ MR1968440, MR1684713,MR1201152}. 
Since the late 1970s, G. Buttazzo and G. Dal Maso have investigated $\Gamma$-convergence in the framework of Lebesgue spaces, Sobolev spaces and $BV$ spaces \cite{MR0583636, MR0794824, MR0567216}. A key tool in the study of $\Gamma$-convergence properties in these frameworks consists in the so-called \emph{integral representation}. By integral representation we mean finding conditions under which a local functional $F(u,A)$ can be expressed e.g. as 
\[
F(u,A)=\int_{A}f(x,u,Du)\, dx
\]
for a suitable \emph{Lagrangian} $f$. In the Euclidean setting this problem is very well understood, and we refer the interested reader to \cite{MR1271411,MR1020296,MR0583636,MR0839727,MR0794824} for a complete overview of the subject. 
In this note, we review some recent results concerning local functionals driven by Lipschitz continuous anisotropies, tailored to the study of anisotropic integral functionals of the form
\begin{equation}\label{introlimitfun}
    F(u,A)=\int_A f(x,u,Xu)\,dx,\qquad A\subseteq\Om,\,u\in W^{1,p}_{X}(\Om),
\end{equation}
where $W^{1,p}_{X}(\Om)$ is an appropriate anisotropic Sobolev space. An \emph{anisotropy} is a family $X=(X_1,\ldots,X_m)$ of Lipschitz continuous vector fields, which induces the anisotropic gradient 
\begin{equation*}
    Xu=\left(X_1u,\ldots,X_mu\right)=X_1u X_1+\cdots+X_muX_m.
\end{equation*}
The term \emph{anisotropy} is motivated by the fact that $X$ may vary point by point.
The investigation of anisotropic variational functionals as in \eqref{introlimitfun}, as well as their related functional frameworks, is originally motivated by L. H\"ormander's seminal work on \emph{hypoelliptic operators} \cite{Hoormander1967147}. 
The latter constitutes a milestone in the study of differential operators with underlying \emph{sub-Riemannian} type structures. 
Important evidences are the works of G.B. Folland \cite{Folland1975161} and of L. Rothschild and E.M. Stein \cite{MR0436223} in the context of \emph{stratified} and \emph{nilpotent Lie groups}. We refer to \cite{MR2363343} for further accounts on analysis on Lie groups, and to \cite{MR3971262,MR1421823,Enrico2025book} for thorough introductions to sub-Riemannian geometry.
As a prototypical example, consider the smooth anisotropy $X=(X_1,X_2)$ on $\R^3$, where 
\begin{equation}\label{anisotropiheisenberg}
    X_1=\frac{\partial}{\partial x_1}+x_2\frac{\partial}{\partial x_3}\qquad\text{and}\qquad X_2=\frac{\partial}{\partial x_2}-x_1\frac{\partial}{\partial x_3}.
\end{equation}
The latter generate, via Lie-brackets, the Lie algebra of the so-called \emph{first Heisenberg group}, and induce on it a \emph{sub-Riemannian structure}. 
The relevant hypoelliptic operator associated with $X$ is the so-called \emph{sub-Laplacian}
\begin{equation}\label{laplaciansub}
    \left(X_1\right)^2u+\left(X_2\right)^2 u,
\end{equation}
and its associated \emph{Dirichlet energy} reads as
\begin{equation} \label{dirienesub}
    \int_\Omega \left(|X_1u|^2+|X_2 u|^2\right)\,{\rm d}x.
\end{equation}
As \eqref{laplaciansub} is not elliptic, \eqref{dirienesub} is not coercive. 
Moreover, the Euclidean Sobolev spaces $W^{1,2}(\Om)$ are not the correct finiteness domains for functionals as in \eqref{laplaciansub}. 
These issues are clear obstructions to the use of classical techniques to study minimization properties of \eqref{dirienesub}, preventing for instance the $L^2$-lower semicontinuity of its Euclidean extension to $L^2(\Om),$ namely
\begin{align*}
    \displaystyle{\begin{cases}
    \int_\Omega \left(|X_1u|^2+|X_2 u|^2\right)\,{\rm d}x&\text{ if }u\in W^{1,2}(\Om),\\
    \infty&\text{ if }u\in L^2(\Omega)\setminus W^{1,2}(\Om).
\end{cases}}
\end{align*}

To overcome these obstacles, building on the foundational work of Folland and Stein \cite{MR0657581}, the correct functional framework has been developed by B. Franchi, R. Serapioni and F. Serra Cassano \cite{MR1437714,MR1448000} and by N. Garofalo and D.M. Nhieu \cite{MR1404326,GN}. 
Precisely, when $1\leq p<\infty$, the study of anisotropic functionals in the greater generality of \eqref{introlimitfun} can be carried out via the introduction of suitable anisotropic Sobolev and $BV$ spaces, say $W^{1,p}_X(\Om)$ and $BV_X(\Om)$, emerging as the natural domains of functionals as in \eqref{introlimitfun}. 
When instead $p=\infty$, the reader is referred to \cite{Capogna2024,PVW,MR2247884,MR2546006} for some anisotropic $L^\infty$-variational problems.\\

The study of anisotropic integral representation results started in \cite{MR4108409,MR4504133}, where the authors investigated integral representation and $\Gamma$-convergence properties of functionals $F(u,A)$ modeled by suitable anisotropies $X=(X_1,\ldots,X_m)$, assuming that
\begin{equation}\label{revisiogiulic}\tag{LIC}
    \text{$X_1(x),\ldots,X_m(x)$ are linearly independent for a.e. $x\in\Om$}
\end{equation}
and
\begin{equation}\label{revisiogiutraslinv}\tag{TI}
    \text{$F(u+c,A)=F(u,A)$ for any function $u$, any set $A$ and any constant $c$.}
\end{equation}
According to \cite{MR4108409}, we  call a functional \emph{translation-invariant} whether it satisfies \eqref{revisiogiutraslinv}, while we say that $X$ satisfies the \emph{linear independence condition} if it satisfies \eqref{revisiogiulic}. Local functionals satisfying \eqref{revisiogiutraslinv} model integral functionals of the form
\begin{equation}\label{traslinvreprintro}
F(u,A)=\int_{A}f(x,Du)\, dx,
\end{equation}
i.e. whose Lagrangian is independent of $u$. Although not fully general, the study of functionals satisfying \eqref{revisiogiutraslinv} is nevertheless highly relevant in the literature. Therefore, we will limit ourselves mainly to this class.
On the other hand, although \eqref{revisiogiulic} already embraces many relevant families of  vector fields studied in literature (cf. \cite{MR4504133} for some instances), we will show how to avoid it. 
More precisely, in \cite{MR4108409} the authors found conditions under which $F$ can be represented as 
\begin{equation}
\label{F_p}F(u,A)=\int_{A}f(x,Xu)\, dx \qquad\text{for any $A\subseteq\Om$ open, for any $u\in \ansol{A}\cap L^p(\Om)$,}
\end{equation}
and for a suitable Lagrangian $f:\Omega\times \mathbf{R}^m\to [0,\infty)$.
Moreover, they applied this characterization to prove a $\Gamma$-compactness theorem for integral functionals of the form \eqref{F_p}, when $ 1< p < \infty$. 
We refer to \cite{MR4054935} for similar results under stronger conditions on the family $\X$.  These results were later extended in \cite{MR4609808,MR4566142} beyond \eqref{revisiogiutraslinv}, to model integral functionals of the form
\[
F(u,A)=\int_{A}f(x,u,Xu)\, dx \qquad\text{for any $A\subseteq\Om$ open, for any $u\in \ansol{A}\cap L^p(\Om)$,}
\]
 and generalizing their Euclidean counterparts \cite{MR0583636,MR0794824}.
Although the above-mentioned contributions relies on the linear independence condition \eqref{revisiogiulic}, the latter has been definitely removed in \cite{withoutlic}, allowing for anisotropic integral representation results in the greatest generality.\\

In this paper we review the approach to the above integral representations, focusing on the case in which \eqref{revisiogiutraslinv} holds. First, we recall the strategy developed by Buttazzo and Dal Maso to prove the Euclidean result. Their approach can be summarized into the following three steps:
\begin{enumerate}
    \item [1.]Prove \eqref{traslinvreprintro} when $u$ is a piecewise affine function. This step require the fact that $F$ is a measure in its second variable.
    \item[2.] Recall that piecewise affine functions are dense in Sobolev spaces.
    \item [3.]Exploit the above steps, together with lower semicontinuity properties of $F$, to extend \eqref{traslinvreprintro} to all Sobolev functions.   
\end{enumerate}
This scheme, which applies \eqref{revisiogiutraslinv}, relies heavily on the second step. However, as noticed in \cite{MR4108409}, it is not generally the case that anisotropic piecewise affine functions approximate anisotropic Sobolev functions. To this aim, the authors of \cite{MR4108409} developed the following alternative strategy.
\begin{enumerate}
	\item[1.] Apply the Euclidean integral representation result to the functional, obtaining an integral representation with respect to a Euclidean Lagrangian $f_e$ of the form
	\begin{equation*}
		F(u,A)=\int_Af_e(x,u,Du)\,dx\qquad\text{for any $A\subseteq\Om$ open, for any $u\in \sol{A}\cap L^p(\Om)$.}
	\end{equation*}
	\item[2.]Find conditions on $f_e$ guaranteeing \color{black} the existence of a Lagrangian $f$ such that
	\begin{equation}\label{primo}
		\int_Af_e(x,u,Du)\,dx=\int_Af(x,u,Xu)\,dx\qquad\text{for any $A\subseteq\Om$ open, for any $u\in C^\infty(A)$.}
	\end{equation}
\item[3.]Combine the previous steps with density properties of smooth functions.
\end{enumerate}
While the third step relies on well-established anisotropic Meyers-Serrin-type approximation results \cite{MR1437714,MR1404326}, the second one constitutes the key of this approach. Indeed, its achievement led the authors of \cite{MR4108409} to assume \eqref{revisiogiulic}, and the author of \cite{withoutlic} to remove this assumption. We point out that, although \eqref{revisiogiulic} is in the end not necessary to establish \eqref{primo}, it nevertheless plays a role in certain secondary issues, including, for example, the uniqueness of representation.\\

Our exposition is organized as follows. In \Cref{Section2} we recall the background about the anisotropic functional framework (\Cref{section21}), introducing the relevant Sobolev spaces (\Cref{secsobspac}) and discussing density properties of piecewise affine and smooth functions (\Cref{secapproxproperties}). In \Cref{seclocfun}, we introduce the basic terminology of local functionals. In \Cref{mainsection} we prove the integral representation results in the \eqref{revisiogiutraslinv} setting. Precisely, first we recall the Euclidean approach (\Cref{seceuclappr}), then we thoroughly explain how to establish \eqref{primo} (\Cref{pseudosection}), and finally we complete the proof of the anisotropic integral representation result (\Cref{secanisores}). In \Cref{seclicnolic} we explain when and how \eqref{revisiogiulic} is actually relevant. Finally, in \Cref{secbeyondti} we guide the reader to existing literature, including cases beyond \eqref{revisiogiutraslinv} and applications to $\Gamma$-convergence.

\section{Anisotropic functional frameworks}\label{Section2}
\subsection*{Main notations}
If no ambiguity arises, we let $\infty=\infty $ and $\N=~\N~\setminus~\{0\}$. If $1\leq p\leq\infty$, we denote by $p'$ the H\"older-conjugate exponent of $p$. 
We fix $m,n\in\N$ such that $0<m\leq n$. 
For $\alpha,\beta\in\N$, we denote by $M(\alpha,\beta)$ the set of matrices with $\alpha$ rows and $\beta$ columns.
We denote by $\langle\cdot,\cdot\rangle$ the Euclidean scalar product, and by $|\cdot|$ its associated norm.
If $L:\R^\alpha\to \R^\beta$ is a linear map, we denote by $\ker(L)\subseteq\R^\alpha$ and $\im(L)\subseteq\R^\beta$ its kernel and its range, respectively.
Fixed an open and bounded set $\Om\subseteq\R^n$, we denote by $\mA$ the class of all the open subsets of $\Om$, and by $\mathcal B$ the class of all the Borel subsets of $\Om$. We let $\mA_0$ be the subfamily of $\mA$ of all the open subsets $A$ of $\Om$ such that $A\Subset\Om$ and by $\mathcal{B}_0$ the subfamily of $\mathcal{B}$ of all the Borel subsets $B$ of $\Om$ such that $B\Subset \Om$.
Finally, we denote by $Du$ the (distributional) Euclidean gradient of $u$. If no ambiguity arises, gradients may be either columns or rows. 
\subsection{Anisotropies} \label{section21}


In the following, $\Om$ is an open and bounded subset of $\R^n$. Given a family $\X=(X_1,\ldots,X_m)$ of Lipschitz continuous vector fields on $\Omega$, such that
\[
X_j=\sum_{i=1}^nc_{j,i}\frac{\partial}{\partial x_i},\quad\text{with }c_{j,i}\in\Lip(\Om)
\]
for any $j=1,\ldots,m$ and any $i=1,\ldots,n$, we denote by $\C(x)\coloneqq[c_{j,i}(x)]_{\substack{{i=1,\dots,n}\\{j=1,\dots,m}}}$ the $m\times n$ \emph{coefficient matrix} associated with $X$. We may refer to $X$ as \emph{anisotropy} or \emph{$X$-gradient}. This notation is motivated by the following definition.
\begin{definition}\label{xgraddef}
Let $u$ and $V$ be a $1$-dimensional and an $m$-dimensional distribution.
\begin{enumerate}
    \item The $X$-gradient of $u$ is the $m$-dimensional distribution defined by 
    \begin{equation*}
        Xu(\varphi)= -u\left(\divv(\C^T\varphi)\right)\qquad\text{for any $\varphi\in \mathbf{C}^\infty _c(\Om;\R^m)$.}
    \end{equation*}
    \item The $X$-divergence of $V$ is the $1$-dimensional distribution defined by 
    \begin{equation*}
        \divv_X(V)(\varphi)= -V(\C D\varphi)\qquad\text{for any $\varphi\in \mathbf{C}^\infty _c(\Om)$.}
    \end{equation*}
\end{enumerate}
\end{definition}
Below are some examples of relevant anisotropies.
\begin{example}[The Euclidean space]
    When $m=n$ and $X_j=\frac{\partial}{\partial x_j}$, we are in the Euclidean isotropic setting, and we feel there is no need to add more.
\end{example}
\begin{example}[Riemannian structures]
    When $m=n$ and $X$ is a global frame of smooth vector fields, $\Om$ can be endowed with a Riemannian metric for which $X$ is an orthonormal frame. Then $X$, $\divv_X$ and $d$ are the Riemannian gradient, divergence and distance respectively.
\end{example}
\begin{example}[Carnot groups]\label{excarnotgroups}
    Let $n=3$, $m=2$, $X_1=\frac{\partial}{\partial x_1}+x_2\frac{\partial}{\partial x_3}$ and $X_2=\frac{\partial}{\partial x_2}-x_1\frac{\partial}{\partial x_3}$. In this case, $X$ is a basis of left-invariant vector fields of the \emph{first Heisenberg group}, a particular Lie group which constitute the first non-trivial instance in the wider class of \emph{Carnot groups}.
\end{example}
\begin{example}[Bracket-generating anisotropies]
  In \Cref{excarnotgroups}, $[X_1,X_2]=-2\frac{\partial}{\partial x_3}$, whence the Lie algebra generated by $X$ has full rank at every point. Vector fields satisfying this property are called \emph{bracket-generating}, and are the building blocks of \emph{sub-Riemannian geometry}.  Not every bracket-generating anisotropy gives rise to a Carnot group. The simplest instance is the \emph{Grushin plane}, i.e. $n=m=2$, $X_1=\frac{\partial}{\partial x_1}$ and $X_2=x_1\frac{\partial}{\partial x_2}$. Indeed, $[X_1,X_2]=\frac{\partial}{\partial x_2}$.
\end{example}
\begin{example}[Carnot-Carathéodory spaces]
    An important tool associated with anisotropies is the so-called \emph{Carnot-Carath\'eodory distance} induced by $X$ on $\Om$ (cf. e.g. \cite{MR1421823}). Precisely, let $x,y\in\Om$, and let $\gamma:[0,1]\to\Om$ be any absolutely continuous curve joining $x$ and $y$. $\gamma$ is \emph{horizontal if}
\begin{equation}\label{horicurvebah}
    \dot\gamma(t)=\sum_{j=1}^ma_j(t)X_j(\gamma(t))\qquad\text{for a.e. $t\in[0,1]$.}
\end{equation}
Then the Carnot-Carathéodory distance is defined by
\begin{equation*}
    d(x,y)=\inf\left\{\int_0^1|(a_1,\ldots,a_m)|\,dt\,:\,\gamma:[0,1]\to\Om,\,\gamma(0)=x,\gamma(1)=y,\,\text{\eqref{horicurvebah} holds}\right\}.
\end{equation*}
Depending on the anisotropy, $d$ may not be always finite. If it is the case, it is a length distance, and $(\Om,d)$ is called \emph{Carnot-Carathéodory space}  \cite{MR0793239}. By \emph{Chow-Rashevskii connectivity theorem} (cf. \cite{MR1880} and \cite[Chapter 3]{Enrico2025book}), bracket-generating anisotropies give rise to finite Carnot-Carathéodory distances. Nevertheless,  this condition is not necessary. As an instance, consider the the family $X=(X_1,X_2)$ of vector fields defined on $\Om=(-1,1)^2\subseteq\rr^2$ by
    \begin{equation*}
        X_1(x)=\frac{\partial}{\partial x_1}\qquad\text{and}\qquad X_2(x)=
\begin{cases}
0&\text{ if }x_1\in(-1,0)\\
x_1\frac{\partial}{\partial x_2}&\text{ if }x_1\in[0,1)
\end{cases}
\,. \qquad\text{  for any $x=(x_1,x_2)\in\Om$.}
    \end{equation*}
 $\Om$ is easily a Carnot-Carathéodory space, while $X$ is not bracket-generating. 
\end{example}
\subsection{Sobolev spaces} \label{secsobspac}
Owing to \Cref{xgraddef}, we introduce the main functional framework.
The anisotropic Sobolev space $ \anso{\Om}$ is defined by
\begin{equation*}
      \anso{\Om}=\{u\in L^p(\Om)\,:\,Xu\in L^p(\Om;\R^m)\}.\\
\end{equation*}
The space $W^{1,p}_{X,loc}(\Om)$ is defined as usual.
It is well-known (cf. \cite{MR0657581}) that $\anso{\Om}$, endowed with the norm
\[
\Vert u\Vert_{\anso{\Om}}\coloneqq\Vert u\Vert_{L^p(\Om)}+\Vert Xu\Vert_{L^p(\Om;\R^m)},
\]
is a Banach space for any $1\leq p<\infty $, reflexive when $1<p<\infty $. 
The Lipschitz continuity assumption on the family $\X$ ensures that 
$W^{1,p}(\Om)$ embeds continuously into $W^{1,p}_X(\Om)$ (cf. \cite{MR4609808,MR4108409}), where $W^{1,p}(\Om)$ denotes the Euclidean Sobolev space.
Precisely, for any $u\in W^{1,p}(\Omega)$, the $X$-gradient admits the Euclidean representation
\begin{equation}\label{euclrapprgrad}
    Xu(x)=\C(x)Du(x)\qquad\text{for a.e. $x\in\Om$.}
\end{equation}
\subsection{Approximation properties}\label{secapproxproperties}
For what concerns density of smooth functions, the classical result of Meyers and Serrin 
is still valid in the anisotropic setting and it was proved, independently, in \cite{MR1437714} and \cite{GN}.
We state the result in what follows, and refer the interested reader to \cite{MR4504133} for further discussions on this topic.
\begin{theorem}\label{MeySer}
For any $1\leq p <\infty$, it holds that $\anso{\Om}=\overline{\mathbf{C}^\infty(\Om)\cap W^{1,p}_X(\Om)}^{\|\cdot\|_{W^{1,p}_X(\Om)}}.$
\end{theorem}
Under additional assumptions on $X$, Poincar\'e inequalities and Rellich-Kondrachov-type theorems may be deduced.
We refer the reader to \cite[Proposition 2.16]{MR4504133} and \cite[Theorem 3.4]{MR1448000}, respectively. Instead, when we come to density of affine functions, the Euclidean and the anisotropic framework differ. In the Euclidean setting, it is well-known that piecewise affine functions approximate Sobolev functions. We recall that  $u:\Om\to\rr$ is \emph{piecewise affine} if  $u\in C^0(\overline{\Om})$ and if there exists a finite family $\{\Om_1,\ldots,\Om_s\}\subseteq\mA$ and a negligible set $N$ such that 
	\begin{equation*}
		\Om=\bigcup_{i=1}^s\Om_i\cup N\qquad\text{and}\qquad u|_{\Om_i}\text{ is affine for any }i=1,\ldots,s.
	\end{equation*}
The following density result holds.
\begin{proposition}\label{3011ektem2}
	Let $1\leq p<\infty$. Let $u\in\so{\Om}$ and let $A'\Subset\Om$. Then there exists a sequence $(u_h)_h$ of piecewise affine functions converging to $u$ strongly in $\so{A'}$.
\end{proposition}
\Cref{3011ektem2} is a corollary of the following result (cf. \cite[Proposition 2.8]{MR1727362}).
\begin{proposition}\label{3011ektem1}
	Let $1\leq p<\infty$. Let $\tilde{\Om}$ be an open and bounded subset of $ \bf R^n $ with Lipschitz boundary. Let $u\in\soz{\tilde{\Om}}$. Then there exists a sequence $(u_h)_h$ of piecewise affine functions on $\tilde{\Om}$ such that, for any $h\in\mathbf{N}$, $u_h|_{\partial\tilde{\Om}}=0$ and converging to $u$ strongly in $\so{\tilde{\Om}}$.
\end{proposition}
Recall that $u$ is affine if and only if $Du$ is constant. If we call a smooth function $u$ \emph{X-affine} as soon as $Xu$ is constant, it is natural to wonder whether the above properties hold for anisotropic Sobolev functions.As shown in \cite{MR4108409}, the answer is negative. The notion of piecewise $X$-affine function is given \emph{verbatim} as above. Let $\Om$ be a bounded open subset of $\rr^3$. Let $X_1,X_2$ be as in \Cref{excarnotgroups}. Let $u(x_1,x_2,x_3)=x_3$. Clearly $u\in C^\infty(\overline\Om)$, and so in particular $u\in W^{1,p}_X(\Om)$ for any $p\geq 1$. An easy computation reveals that a function $v$ is $X$-affine if and only if
    \begin{equation*}
        v(x_1,x_2,x_3)=ax_1+bx_2+c
    \end{equation*}
    for some $a,b,c\in\rr$. Therefore, since $X$-affine functions do not depend on $x_3$, $u$ cannot be pointwise approximated by piecewise $X$-affine functions.
\section{Local functionals and Carathéodory functions}\label{seclocfun}
\subsection{Local functionals}
   
In this section we collect some definitions about increasing set functions and local functionals (cf. \cite{MR1201152}).
The latter, heuristically, behave like variational functionals in the first entry and like measures in the second one. 
\begin{definition}
	We say that $\alpha:\mA\scu [0,\infty]$ is
	\begin{itemize}
	  \item [1.]  \emph{increasing} if 	$\alpha(A)\leq\alpha(B)\text{ for any $A,B\in\mA$ such that $A\subseteq B$;}$
        \item [2.]  \emph{inner regular} if it is increasing and $\alpha(A)=\sup\{\alpha(A')\,:\,A'\Subset A\}\text{ for any $A\in\mA$}$;
        \item[3.]  \emph{subadditive} if it is increasing and
      $            \alpha(A)\leq\alpha(B)+\alpha(C)\text{ for any $A,B,C\in\mA$ with $A\subseteq B\cup C$;}$
      \item[4.]\emph{superadditive} if it is increasing and $		\alpha(C)\geq\alpha(A)+\alpha(B)$ for any $A,B,C\in\mA$ with $A\cap B=\emptyset$ and $A\cup B\subseteq C$;
	\item[5.] a \emph{measure} if it is increasing and it is the restriction to $\mA$ of a Borel measure. 
	\end{itemize}
	\end{definition}
    The next result (cf. \cite[Theorem 14.23]{MR1201152}) gives conditions for a set function to be a measure.
\begin{theorem}\label{2511alfaealfastar}
	Let $\alpha:\mA\scu [0,\infty]$ be an increasing function such that $\alpha(\emptyset)=0$. Define the function $\alpha^*:\mathcal{B}\scu[0,\infty]$ by
	\begin{equation*}
		\alpha^*(B):=\inf\{\alpha(A)\,:\,A\in\mA,\,B\subseteq A\}.
	\end{equation*}
	Then the following conditions are equivalent:
	\begin{itemize}
		\item[1.] $\alpha$ is a measure;
		\item[2.] $\alpha$ is subadditive, superadditive and inner regular;
		\item[3.] $\alpha^*$ is a Borel measure coinciding with $\alpha$ on $\mA$.
	\end{itemize}
\end{theorem}
The next result is an immediate corollary of \Cref{2511alfaealfastar}.
\begin{corollary}\label{2511alfaisradon}
	Let $\alpha:\mA\scu [0,\infty ]$ be a measure. Let $\alpha^*$ be as in \Cref{2511alfaealfastar}. The following conditions are equivalent:
	\begin{itemize}
		\item[1.] For any $A'\in\mA_0$, $\alpha(A')<\infty $;
		\item[2.] $\alpha$ is a Radon measure.  
	\end{itemize}
\end{corollary}

The following definition adapts some standard notation to the anisotropic setting.
\begin{definition}
    If $F:\lp{\Om}\times\mA\scu[0,\infty]$, we say that $F$ is:
	\begin{itemize}
	\item [1.] a \emph{measure} if $F(u,\cdot)$ is a measure for any $u\in\lp{\Om}$;
	\item [2.] \emph{local} if, for any $A\in\mA$ and $u,v\in\lp{\Om}$, then
	$$u|_{A}=v|_{A}\implies F(u,A)=F(v,A);$$
\item [3.] \emph{convex} if $F(\cdot,A)$ restricted to $\ansol{A}\cap L^p(\Om)$ is convex for any $A\in\mA$;
	\item [4.] \emph{$L^p$-lower semicontinuous} (respectively \emph{$W_X^{1,p}$-lower semicontinuous}) if $F(\cdot,A)$ is  $L^p$-lower semicontinuous (respectively $W_X^{1,p}$-lower semicontinuous) for any $A\in\mA$;
	\item [5.]  \emph{weakly-$^\star$ sequentially lower semicontinuous} if $F(\cdot,A)$ restricted to ${\winf{\Om}}$ is weakly-$^\star$ sequentially lower-semicontinuous for any $A\in\mA$.
	\end{itemize}
\end{definition}
We point out that many of the above notions are not relevant in the \eqref{revisiogiutraslinv} setting.
\subsection{Carathéodory functions}
In order to ensure that an integral functional of the form 
\begin{equation*}\label{ifpercaratheodorydef}
    F(u,A)=\int_Af(x,u,Xu)\,dx
\end{equation*}
is well-defined, we impose some conditions on the Lagrangian $f$ in such a way that the function
\begin{equation*}
    x\mapsto f(x,u(x),Xu(x))
\end{equation*}
is integrable on $A$ for any $u\in W^{1,1}_{X,loc}(A)\cap L^p(\Om)$. To this aim, it is customary to work in the class of \emph{Carathéodory functions} (cf. \cite{MR0990890}).
\begin{definition}[Carathéodory functions]
     Let $f:\Om\times \rr\times\rr^m\to  [0,\infty]$ be a function. We say that $f$ is a \emph{Carathéodory function} if:
 \begin{itemize}
 	\item [(i)] $f(\cdot,u,\eta)$ is measurable for any $u\in\rr$ and any $\eta\in\rr^m$;
	\item [(ii)] $f(x,\cdot,\cdot)$ is continuous for a.e. $x\in\Om$.
 \end{itemize}
 \end{definition}
 Carathéodory functions constitute the right class of Lagrangians, as the next proposition shows (cf. e.g. \cite{MR1727362})
 \begin{proposition}
      Let $f:\Om\times \rr\times\rr^m\to  [0,\infty]$ be a Carathéodory function. Then 
      \begin{equation}
    x\mapsto f(x,u(x),Xu(x))
\end{equation}
is measurable for any $A\in\mA$ and any $u\in W^{1,1}_{X,loc}(A)\cap L^p(\Om)$. In particular, being $f$ non-negative, it is integrable on $A$.
 \end{proposition}

\section{Integral representation in the \eqref{revisiogiutraslinv} setting}\label{mainsection}
In this section we show how to prove the anisotropic integral representation result.
\subsection{The Euclidean result}\label{seceuclappr}
As point out in the introduction, the first step consists in exploiting the Euclidean integral representation result. To this aim, we prove it for the sake of completeness. The following theorem has been proved in \cite{MR0794824}, and can be found also in \cite{MR1201152}. \begin{theorem}\label{2611bdmtheoremnotdeponu}
Let $1\leq p<\infty$.
	Let $F:L^p(\Om)\times\mathcal{A}\to [0,\infty]$ be such that:
	\begin{enumerate}
		\item [1.]$F$ is a measure;
		\item [2.]$F$ is local;
		\item [3.]$F$ is translation-invariant;
		\item [4-]There exists $a\in L^1_{loc}(\Omega)$ and $b>0$ such that
        \begin{equation}\label{euclboundinthm}
             F(u,A)\leq \int_{A}a(x)+b|Du|^p\, dx\qquad\text{for any  $A\in\mathcal{A}$ and any $u\in \sol{A}\cap L^p(\Om)$}.
        \end{equation}
		\item [5.] $F$ is $L^p$-lower semicontinuous.
	\end{enumerate}
	Then there exists a Carathéodory function $f_e:\Omega\times\mathbf{R}^n\to [0,\infty)$ such that
	\begin{equation}\label{2611tesibdmnou1}
		\xi\mapsto f_e(x,\xi) \text{ is convex for } a.e.\,x\in\Om,
	\end{equation}
	\begin{equation}\label{2611tesibdmnou2}
		f_e(x,\xi)\leq a(x)+b|\xi|^p\quad\text{for }a.e.\,x\in\Om,\,\text{ for any }\xi\in \bf R^n ,
	\end{equation}
	and 
	\begin{equation}\label{2611tesibdmnou3}
		F(u,A)=\int_{A}f_e(x,Du)\,dx
	\end{equation}
	for any $A\in\mA$ and any $u\in L^p(\Om)\cap W_{loc}^{1,p}(A)$.
\end{theorem}
\begin{proof}
		\textbf{Step 1.} We begin defining $f_e$. Fix $x\in\Om$ and $\xi\in \bf R^n $, and consider the linear function $\varphi_\xi(y):=\langle\xi,y\rangle$. Define 
		$f_e:\Om\times \bf R^n \to [0,\infty]$ by 
		\begin{equation*}
			f_e(x,\xi):=\limsup_{R\to 0^+}\frac{F(\varphi_{\xi},B_R(x))}{|B_R(x)|}.
		\end{equation*}
		By $1,4$ and \Cref{2511alfaisradon}, $F(\varphi_{\xi},\cdot)$ is a Radon measure, absolutely continuous with respect to the Lebesgue measure. By Lebesgue's differentiation theorem, $y\mapsto f_e(y,\xi)$ is in $L^1_{loc}(\Om)$, and
		\begin{equation}\label{2911affinedmnou}
				F(\varphi_{\xi},A)=\int_{A}f_e(y,\xi)\,dy
				=\int_{A}f_e(y,D\varphi_{\xi})\,dy\qquad\text{   for any $A\in\mA$.}
		\end{equation}
             \textbf{Step 2.} We prove \eqref{2611tesibdmnou2}.
	Fix a Lebesgue point $x$ of $a$.
	Let  $\xi\in \bf R^n $ and $\varphi_\xi$ be as above. By $4$,
	\begin{equation*}
		F(\varphi_{\xi},B_R(x))\leq\int_{B_R(x)}a(y)+b|\xi|^p\,dy.
	\end{equation*}
Therefore, \eqref{2611tesibdmnou2} follows dividing by $|B_R(x)|$ and letting $R\to 0^+$. Moreover, up to modifying $f_e$ on a set of measure zero, we can assume that it is finite everywhere. \\
	\textbf{Step 3.} We extend the class of functions for which the integral representation holds. First, by $3$, \eqref{2911affinedmnou} can be extended to every affine function. 
	Let $u$ be a piecewise affine function and $A\in\mA$, and consider the partition $\{\Om'\,\ldots,\Om'_s,N\}$ associated with $u$.
	For any $i=1,\ldots,s$, take an affine function $\varphi_i$ coinciding with $u$ on $\Om_i$. By $1$, $2$ and the absolute continuity of $A\mapsto F(u,A)$ with respect to the Lebesgue measure,
	\begin{equation*} 
		\begin{split}
			F(u,A') 
			 = \sum_{i=1}^{s}F(\varphi_i,A'\cap \Om_i) = \sum_{i=1}^{s}\int_{A'\cap\Om_i}f_e(y,D\varphi_i)\,dy
			 = \int_{A'}f_e(y,Du)\,dy.
		\end{split}
	\end{equation*}
	Therefore
	\eqref{2911affinedmnou} can be extended to every piecewise affine function.\\
\textbf{Step 4.} We prove \eqref{2611tesibdmnou1}. Its proof is known in the literature as \emph{zig-zag lemma} (cf. \cite[Lemma 20.2]{MR1201152}).
	Let $x$ be a Lebesgue point for a. Let $t\in (0,1)$, $\xi_1\neq\xi_2$ in $ \bf R^n $ and $R>0$, and set $\xi:=t\xi_1+(1-t)\xi_2$. By definition of $f_e$, it suffices to show that	\begin{equation*}\label{2611zigzag1}
		F(\varphi_\xi,B_R(x))\leq tF(\varphi_{\xi_1},B_R(x))+(1-t)F(\varphi_{\xi_2},B_R(x))
	\end{equation*}
	when $R$ is small. Define $
		\xi_0:=\frac{\xi_2-\xi_1}{|\xi_2-\xi_1|}$,
	and, for any $h\in\mathbf{N}$ and for any $k\in\mathbf{Z}$, set
	\begin{equation*}
	E^1_{h,k}:=\varphi_{\xi_0}^{-1}\left(\left[\frac{k-1}{h},\frac{k-1}{h}+\frac{t}{h}\right)\right),\quad E^1_h:=\bigcup_{k\in\mathbf{Z}}E^1_{h,k},\quad	\Om^1_{h,k}:=\Om\cap E^1_{h,k},\quad \Om^1_h:=\bigcup_{k\in\mathbf{Z}}\Om^1_{h,k}
	\end{equation*}
    and
	\begin{equation*}
		E^2_{h,k}:=\varphi_{\xi_0}^{-1}\left(\left[\frac{k-1}{h}+\frac{t}{h},\frac{k}{h}\right)\right),\quad E^2_h:=\bigcup_{k\in\mathbf{Z}}E^2_{h,k},\quad	\Om^2_{h,k}:=\Om\cap E^2_{h,k},\quad \Om^2_h:=\bigcup_{k\in\mathbf{Z}}\Om^2_{h,k}
	\end{equation*}
	Since $\xi_0\neq 0$, we claim that 
   \begin{equation}\label{claimdisegni}
        \text{$\chi_{\Om^1_h}\to t$ and $\chi_{\Om^2_h}\to1-t$ weakly-$^\star$ in $L^\infty(\Om)$.}
    \end{equation}
    We just sketch the proof of \eqref{claimdisegni}, focusing on the first property.
    \begin{enumerate}
        \item Let $Q\subseteq \bf R^n $  be closed $n$-cube with one edge parallel to $\xi_0$. Then
\begin{equation}\label{803q}
	|E^1_h\cap Q|\to t|Q|\qquad\text{as $h\to\infty$}.
\end{equation}
\item Let $A\in\mA$. Then
\begin{equation}\label{803secondstep}
	|E^1_h\cap A|\to t|A|\qquad\text{as $h\to\infty$}.
\end{equation}
The proof of \eqref{803secondstep} follows covering $A$ with $n$-cubes as in the previous step, and combining \eqref{803q} with a Vitali covering argument.
\item 
Let $\varphi$ be a piecewise constant function over $\Om$. Then \eqref{803secondstep} implies that
\begin{equation*}\label{803thirdstep}
	\lim_{h\to \infty }\int_{\Om^1_h}\varphi\, dx =t\int_{\Om}\varphi\, dx.
\end{equation*}
\item The claim follows since piecewise constant functions are dense in $L^1(\Om)$. 
    \end{enumerate}
	Define 
	\begin{equation*}
		c^1_{h,k}:=(1-t)\frac{k-1}{h}|\xi_2-\xi_1|\qquad\text{and}\qquad	c^2_{h,k}:=-t\frac{k}{h}|\xi_2-\xi_1|,
	\end{equation*}
	and set
	\begin{equation*}
		u_h(y):=
		\begin{cases}c^1_{h,k}+(\xi_1,y)&\text{ if }y\in\Om^1_{h,k}\\
			c^2_{h,k}+(\xi_2,y)&\text{ if }y\in\Om^2_{h,k}
		\end{cases}
		\,.
	\end{equation*}
	By definition,
	\begin{equation*}
		\begin{split}
			|u_h(y)-\varphi_\xi(y)|\leq \frac{t(1-t)}{h}|\xi_2-\xi_1|
		\end{split}
	\end{equation*}
     for any $y\in\Om$.
	In particular, $u_h\to \varphi_\xi$ uniformly on $\Om$, whence $u_h\to \varphi_\xi$ strongly in $L^p(\Om)$. Since $(u_h)_h$ and $\varphi_\xi$ are piecewise affine, and recalling the previous step,
	\begin{equation*}
		\begin{split}
			F(\varphi_\xi,B_R(x)) & \overset{5}{\leq}\liminf_{h\to \infty} F(u_h,B_R(x))\\
            &=			  \liminf_{h\to \infty}\int_{\Om}f_e(y,Du_h)\,dy\\
	&	\overset{\eqref{claimdisegni}}{=}		 t\int_{\Om^1_h\cap B_R(x)}f_e(y,\xi_1)\,dy+(1-t)\int_{\Om^1_h\cap B_R(x)}f_e(y,\xi_2)\,dy\\
			& = tF(\varphi_{\xi_1},B_R(x))+(1-t)F(\varphi_{\xi_2},B_R(x)).
		\end{split}
	\end{equation*}
    \textbf{Step 5.}
	We show that
	\begin{equation}\label{2911soreprnoubdm}
		F(u,A')=\int_{A'}f_e(x,Du)\,dx
	\end{equation}
	for any $A'\in\mA_0$ and any $u\in W^{1,p}(A')\cap L^p(\Om)$. By \Cref{3011ektem2} there exists a sequence $(u_h)_h$ of piecewise affine functions such that 
	\begin{equation}\label{2911patosobdm}
		u_h\to u\text{ strongly in }\so{A'}.
	\end{equation}
	Moreover, the functional
	\begin{equation*}
		G_{A'}:\left(\left\{u|_{A'}\,:\,u\in\so{A'}\right\},\|\cdot\|_{\so{A'}}\right)\to[0,\infty),\qquad G_{A'}(u):=\int_{A'}f_e(x,Du)\,dx
	\end{equation*}
 is convex and bounded on bounded sets on $\so{A'}$, whence it is continuous.
	Therefore
	\begin{equation} \label{disuginproofeucli}
		F(u,A')\overset{5}{\leq}\liminf_{h\to \infty }F(u_h,A')
		= \liminf_{h\to \infty }\int_{A'}f_e(x,Du_h)\,dx =\int_{A'}f_e(x,Du)\,dx.
	\end{equation}
	We prove the converse inequality.
	Fix $u_0\in\so{A'}$, and define  $H: L^p(\Om)\times\mA\to [0,\infty]$ by $H(u,A):=F(u+u_0,A)$. Then $H$ satisfies all the hypotheses of \Cref{2611bdmtheoremnotdeponu}. Therefore, there exists a Carathéodory function $h:\Omega\times\mathbf{R}^n\to  [0,\infty)$ such that
	\begin{equation}\label{2911auxfunnou0}
		\xi\mapsto h(x,\xi)\text{ is convex for a.e. }x\in\Om,
	\end{equation}
	\begin{equation}\label{2911auxfunnou1}
		h(x,\xi)\leq a_H(x)+b_H|\xi|^p\quad\text{for }a.e.x\in\Om \text{ and any }\xi\in \bf R^m 
	\end{equation}
	for some $a_H\in\lul$ and $b_H>0$, 
	\begin{equation}\label{2911auxfunnou2}
		H(u,A)=\int_{A}h(x,Du)\,dx\qquad\text{ for any $A\in\mA$ and any $u$ piecewise affine},
	\end{equation}
	and
	\begin{equation}\label{2811auxfunnou3}
		H(u,A')\leq\int_{A'}h(x,Du)\,dx\quad\text{ for any $u\in\so{A'}\cap L^p(\Om)$}.
	\end{equation}
	Let $(u_h)_h$ be a sequence of piecewise affine functions converging to $u_0$ strongly in $\so{A'}$. Then $(u_0-u_h)_h$ converges to $0$ strongly in $\so{A'}$. Arguing as before, \eqref{2911auxfunnou0} and \eqref{2911auxfunnou1} imply
	\begin{equation}\label{hfuninproof}
	\int_{A'}h(x,0)\,dx=\lim_{h\to\infty }\int_{A'}h(x,Du_0-Du_h).
	\end{equation}
	Thus we conclude that 
	\begin{equation*} 
		\begin{split}
			\int_{A'}h(x,0)\,dx & \overset{\eqref{2911auxfunnou2}}{=} H(0,A') \\
            &= F(u_0,A')\\
			&\overset{\eqref{disuginproofeucli}}{\leq} \int_{A'}f_e(x,Du_0)\,dx\\
            &=\lim_{h\to\infty }\int_{A'}f_e(x,Du_h)\,dx\\ 
			&= \lim_{h\to\infty }F(u_h,A')
			\\
            &= \lim_{h\to\infty }H(u_0-u_h,A')\\
			&\overset{\eqref{2811auxfunnou3}}{=} \lim_{h\to \infty }\int_{A'}h(x,Du_0-Du_h)\\
            &\overset{\eqref{hfuninproof}}{=}\int_{A'}h(x,0).	
		\end{split}
	\end{equation*}
	In particular, \eqref{2911soreprnoubdm} is proved.\\	
    \textbf{Step 6.}
	Fix $A\in\mA$ and $u\in L^p(\Om)\cap W^{1,p}_{loc}(A)$. Let $A'\in\mA$, $A'\Subset A$. By \eqref{2911soreprnoubdm},
	\begin{equation*} 
		\begin{split}
			F(u,A') = \int_{A'}f_e(x,Du)\,dx.
		\end{split}
	\end{equation*}
Since $F(u,\cdot)$ is a measure, we conclude that 
\begin{equation*}
	\begin{split}
		F(u,A)=\sup\{F(u,A')\,:\,A'\Subset A\}=\sup\left\{\int_{A'}f_e(x,Du)\,dx\,:\,A'\Subset A\right\}=\int_{A}f_e(x,Du)\,dx.
	\end{split}
\end{equation*}
whence \eqref{2611tesibdmnou3} holds.
\end{proof}
\begin{remark}\label{603cambiohpsforsso}
	\Cref{2611bdmtheoremnotdeponu}  still holds if we substitute $5$ with the following two conditions:
	\begin{itemize}
		\item[5$'$.]$F$ is weakly-* sequentially lower semicontinuous;
		\item [5$''$.]$F$ is $W^{1,p}$-lower semicontinuous.
	\end{itemize}
The advantage of requiring weak-* lower semicontinuity and $W^{1,p}$-lower semicontinuity instead of $L^p$-lower semicontinuity is that the formers are both necessary (cf. \cite[Theorem II.1]{MR0751305} for the necessity of weak-$^\star$ sequential lower semicontinuity), while the latter may fail. Indeed, a famous counterexample due to 
Aronszajn (cf. \cite[p. 54]{pauc__1941}) and exploited by Dal Maso in \cite[Example 4.1]{MR0567216}, shows the existence of a Carathéodory function $f_e:(0,1)^2\times\mathbf{R}^2\to [0,\infty)$ such that 
$$\xi\mapsto f_e(x,\xi)\text{ is convex for any }x\in(0,1)^2$$
and 
$$f_e(x,\xi)\leq1+\sqrt{2}|\xi|^p\qquad\text{for any $(x,\xi)\in(0,1)^2\times\mathbf{R}^2$},$$
and such that, if $F$ is its associated integral functional, there exist $u\in\so{(0,1)^2}$ and a sequence $(u_h)_h\subseteq\so{(0,1)^2}$ such that
\begin{equation*}\label{3003contraddiction}
	\lim_{h\to\infty}\Vert u_h-u\Vert_{L^p(\Om)}=0\quad\text{ and }\quad F(u,\Om)>\liminf_{h\to\infty} F(u_h,\Om).
\end{equation*}
\end{remark}
\subsection{From Euclidean to anisotropic representation}\label{pseudosection}
Here we show how to pass from an Euclidean representation, i.e. with respect to an Euclidean Lagrangian $f_e$, to an anisotropic representation, establishing conditions under which \eqref{primo} holds. To this aim, we begin with some linear algebra preliminaries. We refer to \cite{withoutlic} for the results of \Cref{pseudosection}. Following the notation of \cite{MR4609808,MR4108409}, for any $x\in\Om$ we define the linear map $\C(x): \bf R^n \to  \bf R^m $ by
	\begin{equation*}\label{Lx}
		\C(x)(\xi)=\C(x)\cdot\xi
	\end{equation*}
	for any $\xi\in\rr^n$. Moreover, we let
\begin{equation*}\label{N&Vx}
		N_x=\ker(\C(x))\qquad\text{and}\qquad V_x=\left\{\C(x)^T \cdot\eta\ :\ \eta\in\mathbf{R}^m\right\}.
	\end{equation*}
	From standard linear algebra (cf. e.g. \cite{MR0575349}), we know that $ \bf R^n =N_x\oplus V_x$. Hence, for any $x\in\Om$ and $\xi\in \bf R^n $, there are uniqe $\xi_{N_x}\in N_x$ and $\xi_{V_x}\in V_x$ such that
	\begin{equation}\label{splitting}
		\xi=\xi_{N_x}+\,\xi_{V_x}.
	\end{equation}
Therefore, the map $\Pi_x: \bf R^n \to V_x$ defined by $\Pi_x(\xi)=\xi_{V_x}$ is well-posed. The authors of \cite{MR4609808,MR4108409} exploited \eqref{revisiogiulic} to ensure the existence of a right-inverse map associated to $\C(x)$. Precisely, if $X_1(x),\ldots,X_m(x)$ are linearly independent at some $x\in\Om$, then any $\eta\in\rr^m$ can be expressed in the form $\eta=\C(x)\cdot\xi_\eta$ for some $\xi_\eta\in\rr^n$. In the general case, we decompose $\eta\in\rr^m$ as
\begin{equation*}
    \eta=\C(x)\cdot\xi_\eta+\eta^\perp,
\end{equation*}
where $\eta^\perp\in \im(\C(x))^\perp.$ We stress that $\xi_\eta$ is uniquely defined only  modulo $\ker(\C(x))$.
    Let $\Cm:\Om\to  M(n,m)$ be such that $\Cm(x)$ is the \emph{Moore-Penrose pseudo-inverse} of $\C(x)$ (cf. \cite{MR1417720}) for any $x\in\Om$. Precisely, for a fixed $x\in\Om$, $\Cm(x)$ is the unique matrix in $M(n,m)$ such that
    \begin{equation} \label{pseudoinvproperties}
    \begin{aligned}
        &\Cm(x)\cdot\C(x)\cdot\Cm(x)=\Cm(x),\qquad \C(x)\cdot\Cm(x)\cdot\C(x)=\C(x),\\
        &\Cm(x)\cdot\C(x)=\C(x)^T\cdot\Cm(x)^T,\qquad \C(x)\cdot\Cm(x)=\Cm(x)^T\cdot\C(x)^T.
    \end{aligned}
    \end{equation}
The core of our proof is encoded in the following proposition.
\begin{proposition}\label{alglemma}
   For any $x\in\Om$, let $\Cm(x):\rr^m\to  \rr^n$ be the linear map defined by 
    \begin{equation*}
        \Cm(x)(\eta)=\Cm(x)\cdot\eta.
    \end{equation*}
 Then the map
    \begin{equation*}
        x\mapsto\Cm(x)(\eta)
    \end{equation*}
    is measurable for any $\eta\in\rr^m$. Moreover, for any $x\in\Om$, the following facts hold.
    \begin{itemize}
        \item [1.] $\im(\Cm(x))= V_x$.
        \item[2.] $\Pi_x(\xi)=\Cm(x)\cdot\C(x)\cdot\xi$ for any $\xi\in\rr^n$. 
        \item [3.] $\ker(\Cm(x))=\im(\C(x))^\perp$.
    \end{itemize}
\end{proposition}
\begin{proof}
    For a given $\eta\in\rr^n$, it is well-known (cf. \cite{MR1417720}) that 
    \begin{equation*}
        \Cm(x)\cdot\eta=\lim_{h\to\infty }\left(\C(x)^T\cdot\C(x)+\frac{1}{h}I_n\right)^{-1}\cdot\C(x)^T\cdot\eta.
    \end{equation*}
    for any $x\in\Om$. In particular, being $\C$ continuous over $\Om$, $x\mapsto\Cm(x)\cdot\eta$ is the pointwise limit of continuous functions, and hence it is measurable. Now we fix $x\in\Om$. Notice that, by \eqref{pseudoinvproperties},
    \begin{equation*}
        \Cm(x)\cdot\eta=\Cm(x)\cdot\C(x)\cdot\Cm(x)\cdot\eta=\C(x)^T\cdot\left(\Cm(x)^T\cdot\Cm(x)\cdot\eta\right)
    \end{equation*}
    for any $\eta\in\rr^m$, so that $\im(\Cm(x))\subseteq V_x$. To prove the other inclusion, it suffices to show 2. To this aim, fix $\xi\in\rr^n$. by \eqref{splitting} and \eqref{pseudoinvproperties},
    \begin{equation*}
\C(x)\cdot\Cm(x)\cdot\C(x)\cdot\xi=\C(x)\cdot\xi=\C(x)\cdot\left(\Pi_x(\xi)+\xi_{N_x}\right)=\C(x)\cdot\Pi_x(\xi).
    \end{equation*}
   Since we already know that $\Cm(x)\cdot\C(x)\cdot\xi\in V_x$, and being $\C(x)$ injective on $V_x$, 2 follows. To prove 3, fix $\eta\in\ker(\Cm(x))$ and $\xi\in\rr^n$. Then, by \eqref{pseudoinvproperties},
   \begin{equation*}
       \begin{split}
\eta^T\cdot\C(x)\cdot\xi=\eta^T\cdot\C(x)\cdot\Cm(x)\cdot\C(x)\cdot\xi=\left(\Cm(x)\cdot\eta\right)^T\cdot\C(x)^T\cdot\C(x)\cdot\xi=0,
       \end{split}
   \end{equation*}
   so that $\eta\in\im(\C(x))^\perp$. Hence $\ker(\Cm(x))\subseteq\im(\C(x))^\perp$.  Assume by contradiction that there exists $\eta\neq 0$ such that $\eta\in\im(\C(x))^\perp\cap\ker(\C(x))^\perp$. In view of \eqref{pseudoinvproperties},
   \begin{equation}\label{cpinj}
       \Cm(x)\cdot\eta=\Cm(x)\cdot\C(x)\cdot\Cm(x)\cdot\eta.
   \end{equation}
   Since we know that $\ker(\Cm(x))\subseteq\im(\C(x))^\perp$, then $\im(\C(x))\subseteq\ker(\Cm(x))^\perp$, so that both $\eta$ and $\C(x)\cdot\Cm(x)\cdot\eta$ belongs to $\ker(\Cm(x))^\perp$. Being $\Cm(x)$ injective on $\ker(\Cm(x))^\perp$, we conclude from \eqref{cpinj} that $\eta=\C(x)\cdot\Cm(x)\cdot\eta$, a contradiction with $\eta\in\im(\C(x))^\perp$.
\end{proof}

In the next result, we show how to exploit \Cref{alglemma} to obtain \eqref{primo}.
\begin{proposition}\label{anisorappr}
    Let $f_e:\Om\times\rr^n\to [0,\infty]$ be a Carathéodory function.  Assume that 
\begin{equation}\label{0702b}
        f_e(x,\xi)=f_e(x,\Pi_x(\xi))
    \end{equation}
    for a.e. $x\in\Om$ and any $\xi\in\rr^n$. Define the map $f:\Om\times\rr^m\to [0,\infty]$ by
    \begin{equation}\label{fanisotrdef}
        f(x,\eta)=f_e(x,\Cm(x)\cdot\eta)
    \end{equation}
    for any $x\in\Om$ and any $\eta\in\rr^m$. Then $f$ is a Carathéodory function such that
    \begin{equation}\label{costonker}
        f(x,\eta)=f(x,\C(x)\cdot\xi_\eta)
    \end{equation}
    and
    \begin{equation}\label{0702h}
        f_e(x,\xi)=f(x,\C(x)\cdot\xi)
    \end{equation}
    for a.e. $x\in\Om$, any $\eta\in\rr^m$ and any $\xi\in\rr^n$. Moreover, $f$ enjoys the following properties.
    \begin{itemize}
        \item [1.] If there exist $a\in L^1_{loc}(\Om)$ and $b\geq 0$ such that
    \begin{equation}\label{0702g}
     f_e(x,\xi)\leq a(x)+b|\C(x)\cdot\xi|^p\qquad\text{ for a.e. $x\in\Om$ and any $\xi\in\rr^n$,}
    \end{equation}
    then
    \begin{equation}\label{0702i}
        f(x,\C(x)\cdot\xi)\leq a(x)+b|\C(x)\cdot\xi|^p\qquad\text{ for a.e. $x\in\Om$ and any $\xi\in\rr^n$.}
    \end{equation}
     \item [2.] If 
     \begin{equation}\label{feuclconv}
         f_e(x,\cdot)\text{ is convex}\qquad\text{ for a.e. $x\in\Om$,}
     \end{equation}
     then
     \begin{equation}\label{0802conv2}
         f(x,\cdot)\text{ is convex}\qquad\text{ for a.e. $x\in\Om$.}
     \end{equation}
    \end{itemize}
\end{proposition}
\begin{proof}
 Let $f$ be the function in \eqref{fanisotrdef}. First we show that $f$ is a Carathéodory function. To this aim, fix $\eta\in\rr^m$, and define the function $\Phi_{\eta}:\Om\to\rr^n$ by $$\Phi_{\eta}(x)=\Cm(x)\cdot\eta$$
    for any $x\in\Om$. By \Cref{alglemma}, $\Phi_{\eta}$ is measurable. Since 
    \begin{equation}\label{fiscara}
f(x,\eta)=f_e(x,\Phi_{\eta}(x))
    \end{equation}
    for a.e. $x\in\Om$, and being $f_e$ a Carathéodory function, we deduce from \cite[Proposition 3.7]{MR0990890} that $x\mapsto f(x,\eta)$ is measurable for any $\eta\in\rr^m$. Fix $x\in\Om$ and define $\Psi_x:\rr^m\to\rr^n$ by
    $$\Psi_x(\eta)=\Phi_{\eta}(x)$$
    for any $\eta\in\rr^m$. Clearly, $\Psi_x$ is a linear function. In particular, by \eqref{fiscara} and being $f_e$ a Carathéodory function, then $\eta\mapsto f(x,\eta)$ is continuous for a.e. $x\in\Om$, so that $f$ is a Carathéodory function. Moreover, in view of \eqref{feuclconv}, \eqref{fiscara}, the linearity of $\Psi_x$ and the definition of $f$, \eqref{0802conv2} follows. In addition,, \eqref{costonker} follows directly from 3 of \Cref{alglemma}. We prove \eqref{0702h}.
    In view of 2 of \Cref{alglemma}, \eqref{0702b} and the definition of $f$, we infer that
    \begin{equation*}
        \begin{split}
            f(x,\C(x)\cdot\xi)=f_e(x,\Cm(x)\cdot\C(x)\cdot\xi)
            =f_e(x,\Pi_x(\xi))
            =f_e(x,\xi)
        \end{split}
    \end{equation*}
    for a.e. $x\in\Om$ and any $\xi\in\rr^n$, whence \eqref{0702h} follows. Finally, \eqref{0702i} follows by \eqref{0702h} and \eqref{0702g}. 
\end{proof}
\subsection{The anisotropic result}\label{secanisores}
In this section we prove the anisotropic representation result. We refer to \cite{MR4108409} for its proof assuming \eqref{revisiogiulic}, and to \cite{withoutlic} for the general case. As pointed out in the introduction, we cannot rely on the Euclidean approach, since piecewise $X$-affine functions do not approximate anisotropic Sobolev functions. The crucial step is therefore to pass from an Euclidean to an anisotropic representation, establishing \eqref{primo}. In order to rely on \Cref{anisorappr}, we need to ensure that the Euclidean Lagrangian $f_e$ satisfies \eqref{0702b}. To this aim, its convexity will play an important role. However, there are cases beyond the \eqref{revisiogiutraslinv} in which $f_e$ is not convex. This situation is more demanding, and will be treated in \Cref{secbeyondti}.
\begin{theorem}\label{2411modiofmythm24}
	Let $1\leq p<\infty$. Let $F:L^p(\Om)\times\mathcal{A}\to [0,\infty]$ satisfy the following properties:
	\begin{enumerate}
		\item [1.]$F$ is a measure;
		\item [2.]$F$ is local;
        \item [3.]$F$ is translation-invariant;
		\item [4.]There exist $a\in L^1_{loc}(\Om)$ and $b\geq 0$ such that
  \begin{equation}\label{anisobound}
      F(u,A)\leq\int_Aa(x)+b|Xu|^p\,dx \qquad\text{ for any $A\in\mA$ and any $u\in \ansol{A}\cap L^p(\Om)$;}
  \end{equation}
 
  \item [5.]$F$ is $L^p$-lower semicontinuous.
	\end{enumerate}
    
	Then there exists a Carathéodory function $f:\Omega\times\mathbf{R}^m\to  [0,\infty)$ 
	such that	
 \begin{equation}\label{primarappr0802}
     F(u,A)=\int_{A}f(x,Xu)\,d x
 \end{equation}
 for any $A\in\mA$ and any $u\in W^{1,p}_{X,loc}(A)\cap L^p(\Om)$.
 Moreover, $f$ satisfies \eqref{costonker}, \eqref{0702i} and \eqref{0802conv2}.
In addition, if $\tilde f:\Om\times\rr^m\to  [0,\infty )$ is a Carathéodory function which verifies \eqref{costonker}, \eqref{0702i} and for which \eqref{primarappr0802} holds with $\tilde f$ in place of $f$, then
	\begin{equation} \label{2411uniquenessss}
		\tilde f(x,\eta)=f(x,\eta)
	\end{equation}
 for a.e. $x\in\Om$ and any $\eta\in\rr^m$. 
\end{theorem}
\begin{proof}
   \textbf{Step 1.} We show that $F$ verifies the assumptions of \Cref{2611bdmtheoremnotdeponu}.
   We just need to show \eqref{euclboundinthm}. Let $A\in\mA$ and $u\in\so{A}$. Recalling \eqref{euclrapprgrad} and by \eqref{anisobound},
   \begin{equation*}
       F(u,A)\leq \int_Aa(x)+b|Xu|^p\,dx =\int_Aa(x)+b|\C Du|^p\,dx \leq \int_Aa(x)+\tilde b| Du|^p\,dx 
   \end{equation*}
   for some $\tilde b=\tilde b(b,p,X,\Om)\geq 0$. Whence, by \Cref{2611bdmtheoremnotdeponu},
\begin{equation}\label{euclprimarappr0802}
     F(u,A)=\int_{A}f_e(x,Du)\,d x
 \end{equation}
 for any $A\in\mA$ and any $u\in W^{1,p}_{loc}(A)\cap L^p(\Om)$, where $f_e:\Om\times\rr^n\to  [0,\infty)$ is a Carathéodory function satisfying \eqref{0702g} and \eqref{feuclconv}.\\
 \textbf{Step 2.} We show that $f_e$ satisfies \eqref{0702b}. First, we claim that  
 \begin{equation}\label{boundecllagmaan}
     f_e(x,\xi)\leq a(x)+b|\C(x)\xi|^p\qquad\text{for a.e. $x\in\Om$, for any $\xi\in\rr^n$}.
 \end{equation}
 Indeed, let $Q$ be a countable, dense subset of $\rr^n$. Let $\tilde \Om\subseteq\Om$ be such that $|\tilde\Om|=|\Om|$ and every $x\in\tilde\Om$ is a Lebesgue point for $a$ and for $y\mapsto f_e(x,\xi)$ for any $\xi\in Q$. Let $x\in\tilde\Om$ and $\xi\in Q$. Testing \eqref{anisobound} and \eqref{euclprimarappr0802} on $\varphi_\xi(y)=\langle \xi,y\rangle$,
 \begin{equation*}
     \int_{B_R(x)}f_e(x,\xi)\,d x\leq \int_{B_R(x)}a(x)+b|\C(x)\xi|^p\,dx.
 \end{equation*}
Dividing by $|B_R(x)|$ and letting $R\to 0^+$, we conclude \eqref{boundecllagmaan} for a.e. $ x\in\Om$ and for any $\xi\in Q$. Since $f_e$ is a Carathéodory function, \eqref{boundecllagmaan} follows.
 Fix $x\in\Om$ such that \eqref{boundecllagmaan} holds. Fix $\xi\in\rr^n$. Recalling \eqref{splitting}, set $\xi=\xi_{N_x}+\Pi_x(\xi)$. Define on $\ker(\C(x))$ the function   $\tilde \xi\mapsto f(x,\Pi_x(\xi)+\tilde\xi)$ for any $\tilde \xi\in\ker(\C(x))$. It is convex by \eqref{feuclconv}. Moreover,
 \begin{equation*}
     f(x,\Pi_x(\xi)+\tilde\xi)\overset{\eqref{boundecllagmaan}}{\leq} a(x)+b\left|\C(x)\left(\Pi_x(\xi)+\tilde\xi\right)\right|^p=a(x)+b\left|\C(x)\Pi_x(\xi)\right|^p,
 \end{equation*}
 whence it is bounded. Therefore it is constant, and \eqref{0702b} follows.\\
 \textbf{Step 3.}
 By \Cref{anisorappr}, $f:\Om\times\rr^m\to  [0,\infty)$ defined as in \eqref{fanisotrdef} is a Carathéodory function which satisfies \eqref{costonker}, \eqref{0702h}, \eqref{0702i} and \eqref{0802conv2}. Combining \eqref{euclrapprgrad}, \eqref{0702h} and \eqref{euclprimarappr0802},
\begin{equation}\label{primaanisotroreppr}
 F(u,A)=\int_{A}f(x,Xu)\,d x 
\end{equation}
    for any $A\in\mA$ and any $u\in C^\infty(A)\cap L^p(\Om)$. In order to achieve \eqref{primarappr0802}, one argue as in the proof of \Cref{2611bdmtheoremnotdeponu}. The difference here is that we do not rely on the density of piecewise affine functions, but we combine \eqref{primaanisotroreppr} with the Meyers-Serrin-type result provided by \Cref{MeySer}.\\
    \textbf{Step 4.}     
 Finally, assume that there exists a Carathéodory function $\tilde f:\Om\times\rr^m\to  [0,\infty)$ which verifies \eqref{costonker}, \eqref{0702i} and \eqref{primarappr0802}. By \eqref{0702i}, \eqref{primarappr0802} and proceeding as in Step 2, we infer that
	\begin{equation} \label{2411uniquenessssinproof}
		\tilde f(x,\C(x)\cdot\xi)=f(x,\C(x)\cdot\xi)
	\end{equation}
 for a.e. $x\in\Om$ and any $\xi\in\rr^n$. Since both $f$ and $\tilde f$ satisfy \eqref{costonker}, we conclude by \eqref{2411uniquenessssinproof} that
 \begin{equation*} 
		\tilde f(x,\eta)=\tilde f(x,\C(x)\cdot\xi_\eta)=f(x,\C(x)\cdot\xi_\eta)=f(x,\eta)
	\end{equation*}
 for a.e. $x\in\Om$ and any $\eta\in\rr^m$, so that \eqref{2411uniquenessss} follows.
\end{proof}
\section{When \eqref{revisiogiulic} matters}\label{seclicnolic}
For this section, we refer to \cite{withoutlic}. As shown in the previous sections, \eqref{revisiogiulic} does not play any role in the establishment of \Cref{2411modiofmythm24}. However, when \eqref{revisiogiulic} fails, the uniqueness property \eqref{2411uniquenessss} is not in general true without assuming \eqref{costonker}. Moreover, the polynomial bound \eqref{0702i} cannot be replaced by the stronger property
   \begin{equation}\label{strongboundonlylic}
        f(x,\eta)\leq a(x)+b|\eta|^p\qquad\text{ for a.e. $x\in\Om$ and any $\eta\in\rr^m$.}
    \end{equation}
    \begin{example}\label{exsect4}
 Consider the the anisotropy $X=(X_1,X_2)$ defined on $\Om=(0,1)^2\subseteq\rr^2$ by
    \begin{equation*}
        X_1(x)=X_2(x)=\frac{\partial}{\partial x_1}
    \end{equation*}
    for any $x=(x_1,x_2)\in\Om$. Clearly $X_1,X_2$ are Lipschitz continuous on $\Om$ and linearly dependent for any $x\in\Om$. The associated matrices $\C$ and $\Cm$ are respectively
    \begin{equation*}
        \C(x)=\left[ \begin{array}{cc}
	1 & 0  \\
	1 & 0  \\ 
\end{array}
\right ]\qquad\text{and}\qquad\Cm(x)=\left[ \begin{array}{cc}
	1/2 & 1/2  \\
	0 & 0  \\ 
\end{array}
\right ]
    \end{equation*}
for any $x\in\Om$. In particular,
\begin{equation}\label{immaginecontroes}
    N_x=\left\{(0,\lambda)\in\rr^2\,:\,\lambda\in\rr\right\}\qquad\text{and}\qquad \im(\C(x))=\left\{(\lambda,\lambda)\in\rr^2\,:\,\lambda\in\rr\right\}
\end{equation}
  for any $x\in\Om$.
    Consider the functions $f_1,f_2:\Om\times\rr^2\to  [0,\infty)$ defined by 
    \begin{equation*}
     f_1(x,\eta)=2\left(\frac{\eta_1+\eta_2}{2}\right)^2\qquad \text{and}\qquad f_2(x,\eta)=2\left(\frac{\eta_1+\eta_2}{2}\right)^2+e^{(\eta_1-\eta_2)^2}-1
    \end{equation*}
    for any $x\in\Om$ and any $\eta=(\eta_1,\eta_2)\in\rr^2$. They are clearly Carathéodory functions. In view of \eqref{immaginecontroes}, they both verify \eqref{0702i} and \eqref{0802conv2} with $a=0$ and $b=1$. Moreover, 
    \begin{equation}\label{ugualidiverse}
        f_1(x,\C(x)\cdot\xi)=f_2(x,\C(x)\cdot\xi)
    \end{equation}
     for any $x\in\Om$ and any $\xi\in\rr^2$, but they differ otherwise. In particular $f_1$ satisfies \eqref{costonker} and \eqref{strongboundonlylic} with $p=2$, while $f_2$ does not.
    Consider $F_1,F_2:L^2(\Om)\times\mA\to  [0,\infty]$ defined by
\begin{equation}\label{708funzionaleduedef24}
F_j(u,A)=
\displaystyle{\begin{cases}
\int_{A} f_j(x,Xu(x))\,d x&\text{ if }A\in\mA,\,u\in W^{1,2}_{X,loc}(A)\cap L^2(\Om)\\
\infty &\text{otherwise}
\end{cases}
\,.
}
\end{equation}
for $j=1,2$. By means of the forthcoming \Cref{sufficiente}, it is easy to check that $F_1$ and $F_2$ verify the assumptions of \Cref{2411modiofmythm24}, and that
\begin{equation}\label{ugualifinczionali}
    F_1(u,A)=F_2(u,A)=:F(u,A)
\end{equation}
for any $A\in\mA$ and any $u\in L^2(\Om)$. Nevertheless, we know by \eqref{ugualidiverse} that the integral representation of $F$ lacks uniqueness, and moreover that \eqref{strongboundonlylic} is not necessary.
\end{example}
Despite these differences with respect to the \eqref{revisiogiulic} framework, the structural properties of $f$ that one can derive from an integral representation as in \Cref{2411modiofmythm24} are essentially the only ones relevant for deducing structural properties of the associated functional.
\begin{theorem}\label{sufficiente}
    Let $f:\Om\times\rr^m\to[0,\infty]$ be a Carathéodory function. Let $F:L^p(\Om)\times\mA\to [0,\infty]$ be defined as in \eqref{708funzionaleduedef24}.
The following facts hold. 
    \begin{itemize}
        \item [1.] If $f$ satisfies \eqref{0702i}, then
  \begin{equation*}
 F(u,A)\leq\int_Aa(x)+b|Xu|^p\,dx\qquad\text{for any $A\in\mA$ and any $u\in W^{1,p}_{X,loc}(A)\cap L^p(\Om)$.}
  \end{equation*}
     \item [2.] If $\tilde f:\Om\times\rr^m\to [0,\infty ]$ is another Carathéodory function such that 
   \begin{equation*}
       f(x,\C(x)\cdot\xi)=\tilde  f(x,\C(x)\cdot\xi)\qquad\text{for a.e. $x\in\Om$ and any $\xi\in\rr^n$,}
   \end{equation*}
 then
   \begin{equation*}
       F(u,A)=\int_A \tilde f (x,Xu)\,dx\qquad\text{ for any $A\in\mA$ and any $u\in W^{1,p}_{X,loc}(A)\cap L^p(\Om)$.}
   \end{equation*}
  
    \end{itemize}
\end{theorem}
\section{A brief overview of the existing literature}\label{secbeyondti}
In this final section, we guide the reader through the existing literature.
\begin{enumerate}
    \item [1.] For functionals satisfying \eqref{revisiogiutraslinv}, the original Euclidean result, as well as its applications to $\Gamma$-convergence, can be found in \cite{MR0794824}. 
     The anisotropic representation result assuming \eqref{revisiogiulic} and its applications to $\Gamma$-convergence can be found in \cite{MR4108409,MR4504133}. The same results without assuming \eqref{revisiogiulic} have been shown in \cite{withoutlic}. The latter allows to avoid \eqref{revisiogiulic} even beyond the \eqref{revisiogiutraslinv} setting, so we will only refer to it implicitly thereafter. 
      \item [2.] For convex functionals which do not satisfy \eqref{revisiogiutraslinv}, we refer to \cite{MR0583636} for the Euclidean integral representation and $\Gamma$-convergence, and to \cite{MR4609808,MR4566142} for the anisotropic results.
      \item [3.] For non-convex functionals which do not satisfy \eqref{revisiogiutraslinv}, we refer to \cite{MR0839727} for the Euclidean integral representation, and to \cite{MR4609808,MR4566142} for the corresponding anisotropic results and their applications to $\Gamma$-convergence. We point out that, when the approach of \Cref{mainsection} yields an Euclidean Lagrangian $f_e$ which is not convex in any of its variables, in order to deduce \eqref{0702b} we cannot argue as in the second step of the proof of \Cref{2411modiofmythm24}. Instead, a suitable notion of \emph{$X$}-convexity has to be introduced (cf. \cite[Definition 5.3]{MR4609808}). The latter, via a zig-zag argument  (cf. \cite[Lemma 5.5]{MR4609808}), is implied by the $W^{1,p}_X$-lower semicontinuity of the functional, and allows to establish \eqref{0702b} as well  (cf. \cite[Proposition 5.4]{MR4609808}). 
      \item [4.] In the study of $\Gamma$-convergence, it is possible to replace a fixed anisotropy $X$ with a sequence of converging anisotropies, say $(X^h)_h$. We refer the reader to \cite{maioneparonettoverzellesi} for results in this direction.
      \item [5.] Although in a different setting,  we refer to \cite{MR3297726,MR3707084,MR4355973} integral representation and $\Gamma$-convergence result in a Cheeger-Sobolev metric setting.
\end{enumerate}


\begin{thebibliography}{99.}%
%
%
\bibitem{MR0751305}
E.~Acerbi and N.~Fusco.
\newblock Semicontinuity problems in the calculus of variations.
\newblock {\em Arch. Rational Mech. Anal.}, 86(2):125--145, 1984.

\bibitem{MR3971262}
A.~Agrachev, D.~Barilari, and U.~Boscain.
\newblock {\em A comprehensive introduction to sub-{R}iemannian geometry},
  volume 181 of {\em Cambridge Studies in Advanced Mathematics}.
\newblock Cambridge University Press, Cambridge, 2020.

\bibitem{MR1271411}
G.~Alberti.
\newblock Integral representation of local functionals.
\newblock {\em Ann. Mat. Pura Appl. (4)}, 165:49--86, 1993.

\bibitem{MR3297726}
O.~Anza~Hafsa and J.-P. Mandallena.
\newblock On the relaxation of variational integrals in metric {S}obolev
  spaces.
\newblock {\em Adv. Calc. Var.}, 8(1):69--91, 2015.

\bibitem{MR3707084}
O.~Anza~Hafsa and J.-P. Mandallena.
\newblock {$\Gamma$}-convergence of nonconvex integrals in {C}heeger-{S}obolev
  spaces and homogenization.
\newblock {\em Adv. Calc. Var.}, 10(4):381--405, 2017.

\bibitem{MR4355973}
O.~Anza~Hafsa and J.-P. Mandallena.
\newblock Integral representation and relaxation of local functionals on
  {C}heeger-{S}obolev spaces.
\newblock {\em Nonlinear Anal.}, 217:Paper No. 112744, 28, 2022.

\bibitem{MR2363343}
A.~Bonfiglioli, E.~Lanconelli, and F.~Uguzzoni.
\newblock {\em Stratified {L}ie groups and potential theory for their
  sub-{L}aplacians}.
\newblock Springer Monographs in Mathematics. Springer, Berlin, 2007.

\bibitem{MR1968440}
A.~Braides.
\newblock {\em {$\Gamma$}-convergence for beginners}, volume~22 of {\em Oxford
  Lecture Series in Mathematics and its Applications}.
\newblock Oxford University Press, Oxford, 2002.

\bibitem{MR1684713}
A.~Braides and A.~Defranceschi.
\newblock {\em Homogenization of multiple integrals}, volume~12 of {\em Oxford
  Lecture Series in Mathematics and its Applications}.
\newblock The Clarendon Press, Oxford University Press, New York, 1998.

\bibitem{MR1020296}
G.~Buttazzo.
\newblock {\em Semicontinuity, relaxation and integral representation in the
  calculus of variations}, volume 207 of {\em Pitman Research Notes in
  Mathematics Series}.
\newblock Longman Scientific \& Technical, Harlow; copublished in the United
  States with John Wiley \& Sons, Inc., New York, 1989.

\bibitem{MR0583636}
G.~Buttazzo and G.~Dal~Maso.
\newblock {$\Gamma $}-limits of integral functionals.
\newblock {\em J. Analyse Math.}, 37:145--185, 1980.

\bibitem{MR0839727}
G.~Buttazzo and G.~Dal~Maso.
\newblock A characterization of nonlinear functionals on {S}obolev spaces which
  admit an integral representation with a {C}arath\'{e}odory integrand.
\newblock {\em J. Math. Pures Appl. (9)}, 64(4):337--361, 1985.

\bibitem{MR0794824}
G.~Buttazzo and G.~Dal~Maso.
\newblock Integral representation and relaxation of local functionals.
\newblock {\em Nonlinear Anal.}, 9(6):515--532, 1985.

\bibitem{Capogna2024}
L.~Capogna, G.~Giovannardi, A.~Pinamonti, and S.~Verzellesi.
\newblock The asymptotic {$p$}-{P}oisson equation as {$p\to\infty$} in
  {C}arnot-{C}arath\'eodory spaces.
\newblock {\em Math. Ann.}, 390(2):2113--2153, 2024.

\bibitem{MR1880}
W.-L. Chow.
\newblock \"uber {S}ysteme von linearen partiellen {D}ifferentialgleichungen
  erster {O}rdnung.
\newblock {\em Math. Ann.}, 117:98--105, 1939.

\bibitem{MR0990890}
B.~Dacorogna.
\newblock {\em Direct methods in the calculus of variations}, volume~78 of {\em
  Applied Mathematical Sciences}.
\newblock Springer-Verlag, Berlin, 1989.

\bibitem{MR0567216}
G.~Dal~Maso.
\newblock Integral representation on {${\rm BV}( \Omega  )$} of {$\Gamma
  $}-limits of variational integrals.
\newblock {\em Manuscripta Math.}, 30(4):387--416, 1979/80.

\bibitem{MR1201152}
G.~Dal~Maso.
\newblock {\em An introduction to {$\Gamma$}-convergence}, volume~8 of {\em
  Progress in Nonlinear Differential Equations and their Applications}.
\newblock Birkh\"{a}user Boston, Inc., Boston, MA, 1993.

\bibitem{MR0375037}
E.~De~Giorgi.
\newblock Sulla convergenza di alcune successioni d'integrali del tipo
  dell'area.
\newblock {\em Rend. Mat. (6)}, 8:277--294, 1975.

\bibitem{MR0448194}
E.~De~Giorgi and T.~Franzoni.
\newblock Su un tipo di convergenza variazionale.
\newblock {\em Atti Accad. Naz. Lincei Rend. Cl. Sci. Fis. Mat. Nat. (8)},
  58(6):842--850, 1975.

\bibitem{MR1727362}
I.~Ekeland and R.~T\'{e}mam.
\newblock {\em Convex analysis and variational problems}, volume~28 of {\em
  Classics in Applied Mathematics}.
\newblock Society for Industrial and Applied Mathematics (SIAM), Philadelphia,
  PA, english edition, 1999.
\newblock Translated from the French.

\bibitem{MR4609808}
F.~Essebei, A.~Pinamonti, and S.~Verzellesi.
\newblock Integral representation of local functionals depending on vector
  fields.
\newblock {\em Adv. Calc. Var.}, 16(3):767--789, 2023.

\bibitem{MR4566142}
F.~Essebei and S.~Verzellesi.
\newblock {$\varGamma$}-compactness of some classes of integral functionals
  depending on vector fields.
\newblock {\em Nonlinear Anal.}, 232:Paper No. 113278, 21, 2023.

\bibitem{Folland1975161}
G.~Folland.
\newblock Subelliptic estimates and function spaces on nilpotent lie groups.
\newblock {\em Arkiv för matematik}, 13(1):161 – 207, 1975.

\bibitem{MR0657581}
G.~B. Folland and E.~M. Stein.
\newblock {\em Hardy spaces on homogeneous groups}, volume~28 of {\em
  Mathematical Notes}.
\newblock Princeton University Press, Princeton, NJ; University of Tokyo Press,
  Tokyo, 1982.

\bibitem{FonsLeonLp}
I.~Fonseca and G.~Leoni.
\newblock {\em Modern {M}ethods in the {C}alculus of {V}ariations: $L^p$
  {S}paces}.
\newblock Springer Monographs in Mathematics. Springer New York, NY, 2007.

\bibitem{MR1437714}
B.~Franchi, R.~Serapioni, and F.~Serra~Cassano.
\newblock Meyers-{S}errin type theorems and relaxation of variational integrals
  depending on vector fields.
\newblock {\em Houston J. Math.}, 22(4):859--890, 1996.

\bibitem{MR1448000}
B.~Franchi, R.~Serapioni, and F.~Serra~Cassano.
\newblock Approximation and imbedding theorems for weighted {S}obolev spaces
  associated with {L}ipschitz continuous vector fields.
\newblock {\em Boll. Un. Mat. Ital. B (7)}, 11(1):83--117, 1997.

\bibitem{MR1404326}
N.~Garofalo and D.-M. Nhieu.
\newblock Isoperimetric and {S}obolev inequalities for
  {C}arnot-{C}arath\'{e}odory spaces and the existence of minimal surfaces.
\newblock {\em Comm. Pure Appl. Math.}, 49(10):1081--1144, 1996.

\bibitem{GN}
N.~Garofalo and D.-M. Nhieu.
\newblock Lipschitz continuity, global smooth approximations and extension
  theorems for {S}obolev functions in {C}arnot-{C}arath\'{e}odory spaces.
\newblock {\em J. Anal. Math.}, 74:67--97, 1998.

\bibitem{doi:10.1142/5002}
E.~Giusti.
\newblock {\em Direct Methods in the Calculus of Variations}.
\newblock WORLD SCIENTIFIC, 2003.

\bibitem{MR1417720}
G.~H. Golub and C.~F. Van~Loan.
\newblock {\em Matrix computations}.
\newblock Johns Hopkins Studies in the Mathematical Sciences. Johns Hopkins
  University Press, Baltimore, MD, third edition, 1996.

\bibitem{MR1421823}
M.~Gromov.
\newblock Carnot-{C}arath\'{e}odory spaces seen from within.
\newblock In {\em Sub-{R}iemannian geometry}, volume 144 of {\em Progr. Math.},
  pages 79--323. Birkh\"{a}user, Basel, 1996.

\bibitem{Hoormander1967147}
L.~Hörmander.
\newblock Hypoelliptic second order differential equations.
\newblock {\em Acta Mathematica}, 119(1):147 – 171, 1967.

\bibitem{Enrico2025book}
E.~Le~Donne.
\newblock Metric lie groups. carnot-carath\'eodory spaces from the homogeneous
  viewpoint, 2024.
\newblock \url{https://arxiv.org/abs/2410.07291}.

\bibitem{maioneparonettoverzellesi}
A.~Maione, F.~Paronetto, and S.~Verzellesi.
\newblock Variational convergences under moving anisotropies.
\newblock \url{https://doi.org/10.48550/arXiv.2504.02552}, 2025.

\bibitem{MR4108409}
A.~Maione, A.~Pinamonti, and F.~Serra~Cassano.
\newblock {$\Gamma$}-convergence for functionals depending on vector fields.
  {I}. {I}ntegral representation and compactness.
\newblock {\em J. Math. Pures Appl. (9)}, 139:109--142, 2020.

\bibitem{MR4504133}
A.~Maione, A.~Pinamonti, and F.~Serra~Cassano.
\newblock {$\Gamma$}-convergence for functionals depending on vector fields.
  {II}. {C}onvergence of minimizers.
\newblock {\em SIAM J. Math. Anal.}, 54(6):5761--5791, 2022.

\bibitem{MR4054935}
A.~Maione and E.~Vecchi.
\newblock Integral representation of local left-invariant functionals in
  {C}arnot groups.
\newblock {\em Anal. Geom. Metr. Spaces}, 8(1):1--14, 2020.

\bibitem{MR0793239}
A.~Nagel, E.~M. Stein, and S.~Wainger.
\newblock Balls and metrics defined by vector fields. {I}. {B}asic properties.
\newblock {\em Acta Math.}, 155(1-2):103--147, 1985.

\bibitem{pauc__1941}
C.~Pauc.
\newblock {\em “{La}” mèthode métrique en clacul des variations}.
\newblock Actualités scientifiques et industrielles. Hermann, 1941.

\bibitem{PVW}
A.~Pinamonti, S.~Verzellesi, and C.~Wang.
\newblock The {A}ronsson equation for absolute minimizers of supremal
  functionals in {C}arnot--{C}arath\'{e}odory spaces.
\newblock {\em Bull. Lond. Math. Soc.}, 55(2):998--1018, 2023.

\bibitem{MR0436223}
L.~P. Rothschild and E.~M. Stein.
\newblock Hypoelliptic differential operators and nilpotent groups.
\newblock {\em Acta Math.}, 137(3-4):247--320, 1976.

\bibitem{MR0575349}
G.~Strang.
\newblock {\em Linear algebra and its applications}.
\newblock Academic Press [Harcourt Brace Jovanovich, Publishers], New
  York-London, second edition, 1980.

\bibitem{withoutlic}
S.~Verzellesi.
\newblock Variational properties of local functionals driven by arbitrary
  anisotropies.
\newblock \url{https://doi.org/10.48550/arXiv.2402.12227}, 2024.

\bibitem{MR2247884}
C.~Wang.
\newblock The {A}ronsson equation for absolute minimizers of
  {$L^\infty$}-functionals associated with vector fields satisfying
  {H}\"{o}rmander's condition.
\newblock {\em Trans. Amer. Math. Soc.}, 359(1):91--113, 2007.

\bibitem{MR2546006}
C.~Wang and Y.~Yu.
\newblock Aronsson's equations on {C}arnot-{C}arath\'{e}odory spaces.
\newblock {\em Illinois J. Math.}, 52(3):757--772, 2008.
\end{thebibliography}
\end{document}